\documentclass[12pt,reqno]{amsart}
\usepackage{amsmath,amssymb,amsthm,amsfonts,graphicx,color,xcolor}
\usepackage{enumerate}
\usepackage{amssymb,latexsym}
\usepackage{amsfonts}
\usepackage{graphicx}
\usepackage{xcolor}
\usepackage{framed}
\newcommand{\R}{{\mathbb R}}
\newcommand{\C}{{\mathbb C}}

\newtheorem{thm}{Theorem}[section]

\newtheorem{lem}[thm]{Lemma}

\theoremstyle{definition}

\theoremstyle{remark}
\newtheorem{rem}[thm]{Remark}

\numberwithin{equation}{section}

\begin{document}
	\title[Hardy's Theorem for the $(k,\frac{2}{n})-$Fourier Transform]{Hardy's Theorem for the $(k,\frac{2}{n})-$Fourier Transform}
	
	%----------Author 1
	\author[Hanen Jilani]{Hanen Jilani}
	\address[Hanen Jilani]{Universit\'e de Tunis El Manar, Facult\'e des Sciences de Tunis.
		\\ Laboratoire d'Analyse Math\'ematiques et Applications LR11ES11,
		2092 Tunis, Tunisie}
	\email{hanen.jilani@fst.utm.tn}
	\author[Selma Negzaoui]{Selma Negzaoui}
	\address[{Selma Negzaoui}]{Preparatory Institute of Engineering Studies of Monastir, University of Monastir, and Universit\'e de Tunis El Manar, Facult\'e des Sciences de Tunis. Laboratoire d'Analyse Math\'ematique et Applications LR11ES11, 2092 Tunis, Tunisie.
	}
	\email{selma.negzaoui@fst.utm.tn}
	%\thanks{This work was completed with the support of our
	%----------Author 2

	
	\keywords{Hardy's Theorem,  Generalized Fourier transform, Cowling-Price Theorem, Bessel functions, Heat equation, Dynamic version of Hardy's theorem}
	\date{\today}
	\begin{abstract} 
		By comparing a function and its $(k, \frac{2}{n})-$Fourier transform to a Gaussian analogue, $e^{-na|x|^\frac{2}{n}}$,  we establish a Hardy-type uncertainty principle using Phragm\'en-Lindl\"of lemma. Furthermore, we investigate the heat equation in this context, deriving a dynamical version of Hardy's theorem that illustrates the temporal evolution of the uncertainty principle. We also extend our results to $L^p-L^q$ versions, proving Miyachi-type and Cowling-Price-type theorems for the $(k,\frac{2}{n})$-Fourier transform. 
	\end{abstract}
	\maketitle 
	%\tableofcontents
	\section{Introduction}
%%% ----------------------------------------------------------------------
The Heisenberg uncertainty principle states that the position and momentum of a quantum particle cannot be measured simultaneously with arbitrary precision. This fundamental concept has been widely analyzed through the interplay between a function and its Fourier transform. In 1933, Hardy \cite{Ha} demonstrated a striking mathematical manifestation of this principle by showing that a function and its Fourier transform cannot both decay too rapidly compared to the Gaussian, which is optimally localized in both time and frequency domains.
More precisely, let $a$ and $b$ be two positive constants, and let $f$ be a measurable function on $\mathbb{R}$ satisfying
$$\displaystyle{|f(x)|\lesssim e^{-ax^2} \quad\text{ and }\quad |\widehat{f}(x)|\lesssim e^{-bx^2}}.$$
Then $f = 0$ almost everywhere if $ab >\frac{1}{4}$, and $f(x)=C e^{-ax^2}$ for some constant $C$ if $ab = \frac{1}{4}$.
\\
Hardy presented two distinct proofs of his theorem, both involving holomorphic functions and relying on results from complex analysis. The first proof utilizes the Phragm\'en-Lindel\"of principle for entire functions. The second proof also involves entire functions but relies solely on Liouville's theorem, particularly in the case when $ab>\frac{1}{4}$. 
\\
Recently, a significant advancement in understanding this fundamental theorem was achieved by \cite{EKPV}, which provided a proof that does not rely on complex analysis but instead uses real-variable methods by considering Schrödinger evolutions. This result highlights the deep connection between harmonic analysis and partial differential equations. References \cite{CEKPV2010} and \cite{Malinnikova} explore various dynamical versions of the Hardy uncertainty principle for the Fourier transform, including applications to the heat equation.

In this paper, we aim to establish Hardy's theorem in the setting of the one-dimensional $(k,\frac{2}{n})-$generalized Fourier transform, introduced by Ben Sa\"id, Kobayashi, and Ørsted in \cite{BKO}. Their approach provides a deformation of the classical setting by considering the Hamiltonian $$\Delta_{k,a}=\| x\|^ {2-a}\Delta_k -\| x\|^a,$$
where the deformation parameter $a$ is a positive real number arising from the interpolation of minimal unitary representations of two distinct reductive groups, and $\Delta_k$  is the Dunkl Laplacian. In the one-dimensional case, it states as:  
\begin{equation} \label{DL} \Delta_{k} f(x)=f''(x)+\frac{2k}{x}f'(x)-k\frac{f(x)-f(-x)}{x^2},\qquad x \in \R \setminus \{0\}.                                                                                                                  \end{equation}
The generator $\Delta_{k,a}$ allows the construction of a $(k,a)$-generalized Fourier transform, defined by
\[
\mathcal{F}_{k,a}
= e^{\frac{i\pi}{2a}(2\langle k\rangle + N + a - 2)}
   \exp\left(\frac{i\pi}{2a}\,\Delta_{k,a}\right).
\]
The transform $\mathcal{F}_{k,a}$ admits an integral representation involving a kernel $B_{k,a}$ and shares several properties with the classical case, such as the Plancherel formula, the Heisenberg inequality, and others (cf.~\cite{BKO}).

However, many challenging questions remain open, even in the one-dimensional case. For instance, the boundedness of the kernel $B_{k,a}$ and the invariance of the Schwartz space under $\mathcal{F}_{k,a}$, as discussed in \cite{GoIvTi23}, highlight the difficulty of controlling both a function $f$ and its transform $\mathcal{F}_{k,a}(f)$.

The Hardy uncertainty principle provides one result that addresses this issue. However, for the transform $\mathcal{F}_{k,a}$, this principle has so far been established only in the case $a = 2$, which corresponds to the Dunkl transform for arbitrary $k > 0$ (cf.~\cite{GoTr}).

In this work, we address this question for the case $a=\frac{2}{n}$, where $n$ is a positive integer, in one dimension. For clarity, we simplify the notation by writing $\mathcal{F}_{k,n}$ and $B_{k,n}$ instead of $\mathcal{F}_{k,\frac{2}{n}}$ and $B_{k,\frac{2}{n}}$ respectively.

The kernel $B_{k,n}(x, \lambda)$ has the following expression in terms of normalized Bessel functions of indices  $\alpha=kn-\frac{n}{2}$ and $\alpha+n$:
\begin{eqnarray}\label{Bkn}
	B_{k,n}(x,\lambda)=j_{\alpha}(n\vert \lambda x\vert^{\frac{1}{n}})+
	(-i)^n (\frac{n}{2})^n \frac{\Gamma(\alpha+1)}{\Gamma(\alpha+n+1)}
	\lambda x j_{\alpha+n}(n\vert \lambda x\vert^{\frac{1}{n}}).
\end{eqnarray} 
Clearly, the fractional power  $|\lambda x|^\frac{1}{n}$ in the expansion of the kernel \eqref{Bkn} %$B_{k,n}(.,\lambda)$ 
prevents $\mathcal{F}_{k,n}f$,  from being an entire function--a crucial hypothesis for proving Hardy's theorem via complex analysis.
To adress this challenge, we develop two specific deformations of $\mathcal{F}_{k,n}$, denoted by $\mathcal{T}_1$ and $\mathcal{T}_2$ corresponding to the split of even and odd parts (cf. \eqref{T1} and \eqref{T2}).
 By introducing these deformations, we are able to apply a  Phragm\'en-Lindel\"of-type lemma to the new operators and thereby extend Hardy's theorem to the generalized Fourier transform $\mathcal{F}_{k,n}$.
This approach has enabled us to establish the following Hardy-type theorem for the $(k,\frac{2}{n})-$Fourier transform:
\begin{thm}\label{Hardyth}
Let $a$ and $b$ be positive real numbers. Consider a measurable function $f$ on $\mathbb{R}$ satisfying the following inequalities:
\begin{equation}
\vert f(x)\vert \leq C \, e^{-na\vert x\vert^{\frac{2}{n}}},
\label{C1}
\end{equation} and
\begin{equation}
\vert\mathcal{F}_{k,n}f(x)\vert \leq C \, e^{-nb\vert x\vert^{\frac{2}{n}}}.
\label{C2}
\end{equation} Then:
	\begin{enumerate}
		\item[1.]If $ab>\frac{1}{4}$, then $f\equiv 0$.
		\item[2.]If $ab=\frac{1}{4}$, then $f(x)=\lambda\, e^{-na \vert x\vert^{\frac{2}{n}}}$ for some constant $\lambda$.
		\item[3.] If $ab<\frac{1}{4}$,  there exist infinitely many functions satisfying the given conditions.
	\end{enumerate}
\end{thm} 

In our study, it is essential to consider the deformation of the Gaussian,
namely $e^{-na|x|^\frac{2}{n}}$, which arises naturally from the spectral property
\begin{equation}\label{FGaussian}\mathcal{F}_{k,n}\left(e^{-na\vert x\vert^{\frac{2}{n}}} \right) (x)=\frac{1}{(2a)^{k n-\frac{n}{2}+1}}e^{\frac{-n|x|^{\frac{2}{n}}}{4a}}.
\end{equation}
The equality \eqref{FGaussian} served in \cite[Theorem 5.29]{BKO} 
to establish a Heisenberg-type uncertainty principle for the generalized Fourier
transform.
Note that, by considering $n=1$, the optimal function coincides with the Gaussian function $e^{-ax^2}$ and we recover the result  for the Dunkl transform  \cite{GoTr} and for the classical Fourier \cite{Ha}.
Hardy's theorem has been extensively studied in diverse contexts, including those referenced in \cite{SiSu, HaKaNe, SmAb,MejjaoliNegzaoui, CoSiSu}. 
The Phragm\'en-Lindel\"of lemma plays a crucial role in the proofs of those results. However, dynamic versions of the Hardy uncertainty principle for the most recent transformations have not yet been explored.

%To abord dynamic resolution, we need to introduce a Sobolev kind space 

By considering the Dunkl Laplacian $\Delta_k$, defined by \eqref{DL}, the heat operator associated with $(k,\frac{2}{n})-$Fourier transform  is given as follows:
\begin{equation}\label{heatOper}H_{k,n}(t,x):=n\vert x\vert^{2-\frac{2}{n}}\Delta^x_k u(t,x)-\partial_t u(t,x),
\end{equation}
where $x\in \R$ and $t>0$. Here, the superscript in $\Delta^x_k $ 
indicates the relevant variable.
The generalized heat equation stands as
\begin{equation}\label{heateq}H_{k,n}u(t,x)=0.
\end{equation}
Considering the initial condition 
$$u_0(x)=u(0,x)\in L^2_{k,n}(\mathbb{R}),$$ we prove that solving this equation in a Sobolev kind space, is equivalent to a Hardy theorem, providing a dynamical version of the uncertainty principle.

Furthermore, we establish the analogues of Miyachi \cite{Mi} and Cowling-Price \cite{CoPr}, which relax the pointwise Gaussian bound to integrability conditions, that can be considered as $L^p-L^q$ versions of Hardy's theorem.
%We also investigate an $L^p-L^q$ version of Morgan's theorem \cite{MorganGW,BeMo}, which is a refinement of the Hardy's theorem by comparing an $L^p$ function % the decay conditions to $e^{-na|x|^\frac{\gamma}{n}}$, for $\gamma>2$.The sharp condition is given by the equation $(a\gamma)^{\frac{1}{\gamma}}(b\eta)^{\frac{1}{\eta}} =\left(\sin\!\left(\tfrac{\pi}{2}(\eta-1)\right)\right)^{\frac{1}{\eta}},$ where $\eta$ is the H\"older conjugate of $\gamma$, satisfying $\frac{1}{\gamma}+\frac{1}{\eta}​=1$.

The paper is structured as follows:
\\
In Section 2, we present some background information and derive key technical results necessary for our main objective. 
Section 3 is devoted to prove Theorem \ref{Hardyth}, the Hardy-type uncertainty principle. Then, in Section 4, we develop a dynamic version of Hardy's theorem within the framework of the associated heat equation. 
Finally, in Section 5, we extend our results to $L^p -L^q$ settings, proving Miyachi-type and Cowling-Price-type theorems for the  generalized Fourier transform $\mathcal{F}_{k,n}$.

\section{Preliminaries}
%\subsection{$(k,\frac{2}{n})-$Fourier transform and the Gaussian analogue}
Consider $n$ a positive integer, $a=\frac{2}{n}$, and $k\geq \frac{n-1}{2n}$.  The one-dimentional $(k,\frac{2}{n})-$generalized Fourier transform has the integral form:
\begin{equation}\label{Fkn}
\mathcal{F}_{k,n}f (\lambda)=\int_{\R} f(x)B_{k,n}(x,\lambda)d\mu_{k,n}(x),
\end{equation}
where $B_{k,n}(x,\lambda)$ is the kernel given by \eqref{Bkn} and $d\mu_{k,n}(x)$ is the weighted measure defined as:
\begin{equation}\label{measure}
d\mu_{k,n}(x)=\frac{1}{2\Gamma(kn-\frac{n}{2}+1)}\left( \frac{n}{2}\right) ^{kn-\frac{n}{2}}\vert x\vert^{2k+\frac{2}{n}-2}dx.
\end{equation}
For $1 \leq p < +\infty$, let's denote $L^p_{k,n} = L^p (\mathbb{R}, d\mu_{k,n})$, the space of measurable functions on $\R$ satisfying
$$\Vert f\Vert_{L^p_{k,n}}=\left( \int_{\mathbb{R}}\vert f(x)\vert^p d\mu_{k,n}(x) \right)^{\frac{1}{p}}<+\infty, $$
and for $p=+\infty$, $$\Vert f\Vert_{L^\infty_{k,n}}= \text{ess}\sup_{x\in\mathbb{R}}\vert f(x)\vert=\inf \{M \geq 0 ;\, |f| \leq M\;\; \mu_{k,n} a.e.\}. $$
 It is worth mentioning that the kernel $B_{k,n}$ is bounded on $\R$. This allows us to conclude that there exists a positive constant $C(k,n)=\|B_{k,n}\|_\infty$ satisfying, for all $f\in L^1_{k,n}$,
\begin{equation}
\Vert \mathcal{F}_{k,n}f\Vert_{L^\infty_{k,n}} \leq C(k,n)\Vert f\Vert_{L^1_{k,n}}.
\label{Fourier borne}
\end{equation}

It is well knoun from \cite{BKO} that the operator 
$\mathcal{F}_{k,n}: L^2_{k,n}\longrightarrow L^2_{k,n}$
is unitary and satisfies the Plancherel formula:
\begin{equation}
 \Vert \mathcal{F}_{k,n}f\Vert_{L^2_{k,n}}=\Vert f\Vert_{L^2_{k,n}}.
\end{equation}
The same reference, \cite[Theorem 5.3]{BKO}, provides the inversion formula as follows:
\begin{equation}\label{Finverse}
\mathcal{F}^{-1}_{k,n}f(x)=\mathcal{F}_{k,n}f((-1)^nx), \qquad x\in \mathbb{R}. 
\end{equation}
Regarding \eqref{Bkn}, while the even part of the kernel $B_{k,n}$ is provided by the normalized Bessel function $j_\alpha$ of index $\alpha$: 
\begin{equation}\label{j=J}j_\alpha(z)=2^\alpha \Gamma(\alpha+1)z^{-\alpha}J_\alpha(z)=\Gamma(\alpha+1)\sum_{k=0}^\infty \frac{(-1)^k}{k! \Gamma(\alpha+k+1)}\left(\frac{z}{2} \right)^{2k},
\end{equation}
the odd part can, after making an appropriate substitution, be expressed in terms of  
$u^n j_{\alpha+n}(u)$. This splitting has served to establish a product formula for $B_{k,n}$ in \cite{Boubatra_Negzaoui_Sifi}. Indeed it was shown that the product of two kernels can be written as an integral of the kernel with weighted measure involving Gegenbauer polynomials  of ordre $\alpha$ and degree $n$, 
\begin{equation}\label{Gegenbauer}C_n^{(\alpha)}(t)=\frac{1}{\Gamma(\alpha)}\sum_{k=0}^{\left[\frac{n}{2}\right] }(-1)^k \frac{\Gamma(n-k+\alpha)}{k!(n-2k)!}(2t)^{n-2k}.
\end{equation}
By virtue of \cite[$(3)$ p.48]{Wa}, for $\alpha>-\frac{1}{2}$, the normalized Bessel function $j_\alpha$ has the Poisson integral representation 
\begin{equation}\label{BesselPoisson} j_{\alpha}(x)=2\frac{\Gamma(\alpha+1)}{\sqrt{\pi}\Gamma(\alpha+\frac{1}{2})} \int_{0}^{1} (1-t^{2})^{\alpha-\frac{1}{2}} \cos(xt)dt.
\end{equation}
For the odd part, the Gegenbauer's generalisation of Poisson's integral, \cite[$(3)$, p. 50]{Wa}, states as
\begin{equation}
	\label{GegenbauerBesselPoisson}
	J_{\alpha+n}(z)=\frac{(-i)^n n!\Gamma(2\alpha)\left( \frac{1}{2}z\right)^\alpha}{\sqrt{\pi}\Gamma(\alpha+\frac{1}{2})\Gamma(2\alpha+n)} \int_{-1}^{1}C_n^{\alpha}(t)(1-t^2)^{\alpha-\frac{1}{2}}e^{izt}dt.
\end{equation} 
Or equivalently, 
\begin{equation}\label{unj}
u^n j_{\alpha+n }(u)=a_{\alpha,n} \int_{0}^{1}C_n^{(\alpha)}(t)(1-t^2)^{\alpha -\frac{1}{2}}\cos(ut+n\frac{\pi}{2})\, dt,
\end{equation}
where the constant $a_{\alpha,n}$ is given by \begin{equation}\label{aan}a_{\alpha,n}=\frac{2^{2\alpha+n}n!}{\pi}(\alpha+n)\beta(\alpha,\alpha+n),
\end{equation} and $C_n^{(\alpha)}$ is the Gegenbauer polynomial of ordre $\alpha$ and degree $n$ given by \eqref{Gegenbauer}. 

Denote $f_e$ and $f_o$ the even and the odd parts of $f$, respectively: $$\forall\, x\in\R,\qquad f_e(x)=\frac{f(x)+f(-x)}{2}\quad and \quad f_o(x)=\frac{f(x)-f(-x)}{2}.$$ Consider the transformations $\mathcal{T}_1 $ and $\mathcal{T}_2 $, which act on the space $L^1_{k,n}$ as follows:   
\begin{equation}\label{T1}
\mathcal{T}_1 f(z)=\int_{\mathbb{R}} f_e(u)j_{k n-\frac{n}{2}} (nz\vert u\vert^{\frac{1}{n}})d\mu_{k,n}(u),                                         \end{equation}
and \begin{equation}\label{T2}
\mathcal{T}_2 f(z)=\frac{(-i)^n \Gamma(k n-\frac{n}{2}+1)}{\Gamma(k n+\frac{n}{2}+1)}\left( \frac{n}{2}\right)^n \int_{\mathbb{R}}f_o(u)u z^n j_{k n+\frac{n}{2}} (nz\vert u\vert^{\frac{1}{n}})d\mu_{k,n}(u).\end{equation}
These transformations are closely related to  $\mathcal{F}_{k,n}$. Indeed, for all $x\in\R$,
 \begin{equation} \mathcal{F}_{k,n}f(x^n)=\mathcal{T}_1 f_e(x)+\mathcal{T}_2 f_o(x). 	\label{reel tildea} \end{equation}
 For these transformations, we have the following lemmas that ensure the analyticity of $\mathcal{T}_1f$ and $\mathcal{T}_2f$ in the complex plane $\C$. 
\begin{lem}\label{entire1}
 Let $p\in [1,+\infty]$, $ a>0$, and $f$ be a measurable function on $\mathbb{R}$ satisfying 
\begin{equation}
e^{na\vert x \vert^\frac{2}{n}}f\in L^p_{k,n}.
\label{Cond.1}
\end{equation}
Then the function $\mathcal{T}_1f$ is well-defined and entire on $\mathbb{C}$. Furthermore, $\mathcal{T}_1f$ 
satisfies the following inequality:
\begin{equation}
\forall z\in \C, \qquad \vert \mathcal{T}_1 f(z)\vert \leq C({k,n,p})\, e^{\frac{n}{4a}\Im(z)^2},
\label{Tineg}
\end{equation}
where  $C({k,n,p})$ is a positive constant depending only on $k$, $n$ and $p$.
\end{lem}
\begin{proof}
Consider  a function $f$ satisfying \eqref{Cond.1}. Note that the mapping $z\mapsto f_e(u) j_{k n-\frac{n}{2}} (nz\vert u\vert^{\frac{1}{n}})$ defines an entire function on $\mathbb{C}$. For all  $z\in \C$ and $t\in [0,1]$, we have:
\begin{eqnarray}\label{cos_ineg}
	\vert \cos(nz\vert u\vert^\frac{1}{n}t)\vert=\left\vert\frac{e^{inz\vert u\vert^\frac{1}{n}t}+e^{-inz\vert u\vert^\frac{1}{n}t}}{2}\right\vert	&\leq& e^{n\vert \Im(z)\vert \vert u\vert^\frac{1}{n}}.
\end{eqnarray}
Using the fact that $$\int_{0}^{1}(1-t^2)^{kn-\frac{n}{2}-\frac{1}{2}}dt=\beta(\frac{1}{2},kn-\frac{n}{2}+\frac{1}{2})= \frac{\sqrt{\pi}\Gamma(kn-\frac{n}{2}+\frac{1}{2})}{2\Gamma(kn-\frac{n}{2}+1)},$$
we can derive from \eqref{BesselPoisson} that
\begin{equation}\label{domt1}
 \vert j_{kn-\frac{n}{2}}(nz\vert u\vert^{\frac{1}{n}})\vert \leq e^{n\vert \Im(z)\vert\vert u\vert^{\frac{1}{n}}}.
\end{equation}
Thus, for all $R>0$ and for all $z\in\C$ with $|\Im(z)|\leq R$, we have $$\left|f_e(u)j_{k n-\frac{n}{2}} (nz\vert u\vert^{\frac{1}{n}})|u|^{2k+\frac{2}{n}-2}\right|\leq \varphi_R(u)= \left|f_e(u)e^{nR\vert u\vert^{\frac{1}{n}}}|u|^{2k+\frac{2}{n}-2}\right|. $$
It suffices to show that $\varphi_R$ belongs to $L^1(\R)$  in order to conclude that $\mathcal{T}_1f$ is well-defined and entire on  $\C$. This can be demonstrated using Hölder's inequality along with the condition \eqref{Cond.1} as follows:
$$\|\varphi_R\|_{L^1(\R)}=\int_{\mathbb{R}}\vert f_e(u)\vert e^{nR\vert u\vert^\frac{1}{n}}d\mu_{k,n}(u)\leq \Vert e^{n a\vert \, .\,\vert^\frac{2}{n}} f_e \Vert_{L_{k,n}^p} \Vert e^{-n a\vert\, .\,\vert^\frac{2}{n}+nR\vert \, .\,\vert^\frac{1}{n}} \Vert_{L_{k,n}^{p'}}<+\infty,$$
where $\frac{1}{p}+\frac{1}{p'}=1$.

Now, let's prove inequality \eqref{Tineg}.\\
Applying the Poisson integral representation \eqref{BesselPoisson} and performing a change of variables, we obtain for $\alpha=kn-\frac{n}{2}$, %\begin{equation}\mathcal{T}_1f(z)=\frac{2n}{\sqrt{\pi}\Gamma(\alpha+\frac{1}{2})}\left( \frac{n}{2}\right)^\alpha \int_{0}^{\infty}\int_0^t f_e(t^n)(t^2-u^2)^{\alpha-\frac{1}{2}}\cos(nuz) \, du\,t\, dt.\label{T_1fe_inq}\end{equation} 
%Fubini's theorem and equation \eqref{cos_ineg} yield 
\begin{equation}\label{T_1fe_inq2} 
		\vert \mathcal{T}_1f(z)\vert \leq\frac{2n}{\sqrt{\pi}\Gamma(\alpha+\frac{1}{2})}\left( \frac{n}{2}\right)^\alpha \int_{0}^{+\infty}\int_u^{+\infty} \vert f_e(t^n)\vert (t^2-u^2)^{\alpha-\frac{1}{2}} \; t\,dt\;e^{n u \vert \Im(z)\vert }\; du.
	\end{equation} 
Let's denote the integral
$$I(u)= \int_{u}^{+\infty}\vert f_e(t^n)\vert (t^2-u^2)^{\alpha-\frac{1}{2}}t\,dt.$$ Note that \begin{equation}\label{eg}
nu \vert \Im(z)\vert=\frac{n}{4a}\vert \Im(z)\vert^2+anu^2 -an(u-\frac{\vert\Im(z)\vert }{2a})^2,
\end{equation} 
then applying Hölder's inequality, we obtain for $\frac{1}{p}+\frac{1}{p'}=1$,
\begin{equation*}
	\vert \mathcal T_1 f(z)\vert \leq\frac{2n}{\sqrt{\pi}\Gamma(\alpha+\frac{1}{2})}\left(\frac{n}{2} \right)^\alpha
	e^{\frac{n}{4a}\Im(z)^2}	 \left( \int_{0}^{+\infty}e^{napu^2}\left[ I(u)\right]^p\;du\right)^\frac{1}{p} \left( \int_{0}^{+\infty} e^{-p'an(u-\frac{\vert\Im(z)\vert }{2a})^2}du\right)^{\frac{1}{p'}}.
\end{equation*}
Making a change of variable, one can easily find when $p'\in[1,+\infty)$, that
$$
	\int_{0}^{+\infty} e^{-p'an(u-\frac{\vert\Im(z)\vert }{2a})^2}du=	\int_{-\frac{\vert\Im(z)\vert}{2a}}^{+\infty} e^{-p'anx^2}dx\leq \int_{\R} e^{-p'anx^2}dx=\sqrt{\frac{\pi}{p'an}}.
$$
Consequently, 
\begin{equation}\label{T1fI}\vert \mathcal T_1 f(z)\vert \leq C(k,n,p) \, e^{\frac{n}{4a}\Im(z)^2} \left( \int_{0}^{+\infty}e^{napu^2}\left[ I(u)\right]^pdu\right)^\frac{1}{p}.
\end{equation}
H\"older's inequality leads to
$$I(u)\leq \left[ \int_{u}^{+\infty}e^{anpt^2}\vert f_e(t^n)\vert^p (t^2-u^2)^{\alpha-\frac{1}{2}}t\,dt\right]
 \left[ \int_{u}^{+\infty}e^{-anp't^2} (t^2-u^2)^{\alpha-\frac{1}{2}}t\,dt\right]^\frac{p}{p'}.$$
Making a change of variable allows us to compute the last integral as:
\begin{equation}
	\int_u^{+\infty} e^{-np'at^2} (t^2-u^2)^{\alpha-\frac{1}{2}} t\, dt=\frac{e^{-nap'u^2}}{2(anp')^{\alpha+\frac{1}{2}}}\int_0^{+\infty} e^{-x} x^{\alpha-\frac{1}{2}} dx=\frac{\Gamma(\alpha+\frac{1}{2})}{2(anp')^{\alpha+\frac{1}{2}}}e^{-anp'u^2},
	\label{sup_int_u}
\end{equation}
which gives 
$$\left[I(u)\right]^p\leq C \left[ \int_{u}^{+\infty}e^{anpt^2}\vert f_e(t^n)\vert^p (t^2-u^2)^{\alpha-\frac{1}{2}}t\,dt\right]
 \left[ e^{-anpu^2}\right].$$
Hence, by \eqref{T1fI} 
$$\vert \mathcal T_1 f(z)\vert^p \leq C(k,n,p,a) \, e^{\frac{np}{4a}\Im(z)^2}  \int_{0}^{+\infty}\left( \int_{u}^{+\infty}e^{anpt^2}\vert f_e(t^n)\vert^p (t^2-u^2)^{\alpha-\frac{1}{2}}t\,dt\right)du.$$
Using Fubini's theorem 
$$\vert \mathcal T_1 f(z)\vert^p \lesssim e^{\frac{np}{4a}\Im(z)^2}  \int_{0}^{+\infty}\left( \int_{0}^{t}(t^2-u^2)^{\alpha-\frac{1}{2}}du\right)  \vert f_e(t^n)\vert^p e^{anpt^2} t\,dt, $$ which leads to % An additional change of variable yields 
\begin{align*}
\vert \mathcal T_1 f(z)\vert^p \lesssim  e^{\frac{np}{4a}\Im(z)^2}  \int_{0}^{+\infty} \vert f_e(t^n)\vert^p e^{anpt^2} t^{2kn-n+1}dt.
\end{align*}
Finally, the fact that $$\int_{0}^{+\infty} \vert f_e(t^n)\vert^p e^{anpt^2} t^{2kn-n+1}dt\lesssim \left\| e^{an\vert .\vert^\frac{2}{n}}f\right\|_{L^p_{k,n}}$$ allows to deduce \eqref{Tineg} for $1<p<+\infty$.

\textbf{The case $p=1$}: It follows from \eqref{domt1} and 
\eqref{eg} that
\begin{align*}
	\vert \mathcal{T}_1f(z)\vert \lesssim \Vert e^{na\vert .\vert^{\frac{2}{n}}}f\Vert_{L^1_{k,n}}\;e^{\frac{n}{4a}\vert \Im(z)\vert^2}.
\end{align*}
\textbf{In the case $p=+\infty$}, we obtain:
	$$I(u)\leq \Vert e^{na\vert .\vert^{\frac{2}{n}}}f\Vert_{L^\infty_{k,n}}\int_u^{+\infty} e^{-nat^2} (t^2-u^2)^{\alpha-\frac{1}{2}} t\, dt$$
As we have \eqref{sup_int_u}, one can easily deduce  
\begin{equation*}
		I(u)\lesssim \Vert e^{na\vert .\vert^{\frac{2}{n}}}f\Vert_{L^\infty_{k,n}}
		 e^{-nau^2}	
\end{equation*}
Thus, \eqref{T_1fe_inq2} leads to
$$ \vert\mathcal{T}_1f_e(z)\vert\lesssim\Vert e^{na\vert .\vert^{\frac{2}{n}}}f\Vert_{L^\infty_{k,n}} e^{\frac{n}{4a}\vert \Im(z)\vert^2} \int_{0}^{+\infty}e^{-na(u-\frac{\vert \Im(z)\vert}{2a})^2} du.$$
Which proves the wanted inequality.
%\begin{equation}	\vert\mathcal{T}_1f_e(z)\vert\lesssim\Vert e^{na\vert .\vert^{\frac{2}{n}}}f\Vert_{L^\infty_{k,n}}\; e^{\frac{n}{4a}\vert \Im(z)\vert^2}.\end{equation}
\end{proof}
We proceed similarly to prove the result for the second transformation $\mathcal{T}_2$.
\begin{lem}\label{entire2}
 Let $p\in [1,+\infty]$ and $ a>0$. Consider $f$ a measurable function on $\mathbb{R}$ verifying relation \eqref{Cond.1}. Then the function $\mathcal{T}_2f$ is well-defined and entire on $\mathbb{C}$. Furthermore, $\mathcal{T}_2f$ 
satisfies the following inequality:
\begin{equation}
\forall z\in \C, \quad \vert \mathcal{T}_2 f(z)\vert \leq C({k,n,p})\, e^{\frac{n}{4a}\Im(z)^2},
\label{T2neg}
\end{equation}
where  $C({k,n,p})$ is a positive constant depending only on $k$ and $n$.
\end{lem}
\begin{proof} 
%According to Lemma \ref{lemmaunj}, for $\alpha=kn-\frac{n}{2}$, we have $$n|u|z^n j_{\alpha+n }(nz|u|^\frac{1}{n})=a_{\alpha,n} \int_{0}^{1}C_n^{(\alpha)}(t)(1-t^2)^{\alpha -\frac{1}{2}}\cos(nz|u|^\frac{1}{n}t+n\frac{\pi}{2})\, dt.$$ 
Since $ \cos(nz\vert u\vert^\frac{1}{n}t+n\frac{\pi}{2})$ can be seen as $\pm\cos(nz\vert u\vert^\frac{1}{n}t)$ or  $\pm  \sin(nz\vert u\vert^\frac{1}{n}t)$, then in both cases, for all $z\in \C$, and for all $t\in [0,1]$, we have
\begin{equation}\label{cos2}\left\vert \cos(nz\vert u\vert^\frac{1}{n}t+n\frac{\pi}{2})\right\vert \leq e^{n\vert \Im(z)\vert \vert u\vert^\frac{1}{n}}.
\end{equation}
Note that, for $\alpha>-\frac{1}{2}$, the Gegenbauer polynomials $C^{(\alpha)}_n$ are bounded on $[0,1]$, which leads to the following upper bound
\begin{equation}\label{domt2}
\vert u z^n j_{k n+\frac{n}{2}} (nz\vert u\vert^{\frac{1}{n}}) \vert \leq  C(k,n)\, e^{n\vert \Im(z)\vert \vert u\vert^\frac{1}{n}},
\end{equation} where $C(k,n)$ denotes a constant depending on $k$ and $n$.  Hence, if we consider $R>0$ then, for all $z\in\C$ with $|\Im(z)|\leq R$, $$\left|f_o(u)nu z^n j_{k n+\frac{n}{2}} (nz\vert u\vert^{\frac{1}{n}})\vert|u|^{2k+\frac{2}{n}-2}\right|\leq \psi_R(u)= C(k,n)\left|f_o(u)e^{nR\vert u\vert^{\frac{1}{n}}}|u|^{2k+\frac{2}{n}-2}\right|. $$ 
Similar argument as for $\varphi_R$ ensures that $\psi_R$ belongs to $L^1(\R)$, together with the fact that the mapping $z\longmapsto f_o(u)u z^n j_{k n+\frac{n}{2}} (nz\vert u\vert^{\frac{1}{n}})$ is an entire function on $\mathbb{C}$, prove that $\mathcal{T}_2f_0$ is well defined and entire on $\C$.

To prove \eqref{T2neg}, note that $xf_o(x)$ is an even function. Then %a change of variable and Lemma \ref{lemmaunj} lead to
$$ \mathcal{T}_2f(z)=C(k,n) \int_{0}^{\infty}f_o(t^n)\int_0^1 C_n^{kn-\frac{n}{2}}(s)(1-s^2)^{kn-\frac{n}{2}-\frac{1}{2}}\cos(nzts+n\frac{\pi}{2})ds\, t^{2kn-n+1} dt.$$ Therefore, applying a change of variable, 
$$ \left\vert \mathcal{T}_2f(z)\right\vert\lesssim \int_{0}^{\infty}|f_o(t^n)|\int_0^t (t^2-u^2)^{kn-\frac{n}{2}-\frac{1}{2}}|\cos(nzu+n\frac{\pi}{2})|du\, t\,dt.$$ 
Fubini's theorem, and inequality \eqref{cos2}, we obtain:
\begin{equation}
\vert	\mathcal{T}_2f(z)\vert \lesssim \int_{0}^{\infty} \int_u^{+\infty}\vert f_o(t^n)\vert (t^2-u^2)^{kn-\frac{n}{2}-\frac{1}{2}}\; t\,dt\;e^{n\vert u\vert \vert \Im(z)\vert} \; du.
\label{T_2eq}
\end{equation} 
Here we recognize similar integral as provided in the inequality \eqref{T_1fe_inq2}, the only difference resides in considering $f_o$ instead of $f_e$. Since, for all $p\in [1,+\infty]$, 
 $\Vert e^{an\vert.\vert^\frac{2}{n} }f_o\Vert_{L^p_{k,n}}\leq \Vert e^{an\vert.\vert^\frac{2}{n} }f\Vert_{L^p_{k,n}},$
similar discussions as in the proof of Lemma \ref{entire1} allows to derive \eqref{T2neg}
%$$ \vert \mathcal{T}_2f_o(z)\vert \lesssim \Vert e^{an\vert.\vert^\frac{2}{n} }f\Vert_{L^p_{k,n}}\; e^{\frac{n}{4a}\vert \Im(z)\vert^2}.$$ 
 \end{proof}
\section{Hardy theorem}
Before adressing the proof of our main result, we recall the Phragm\'en-Lindel\"of type lemma (cf. \cite{GoTr, HaKaNe}). 
\begin{lem}[Phragm\'en-Lindel\"of]
Let $p\in[1,+\infty]$ and $h$ be an entiere function on $\mathbb{C}$. We assume that
	\begin{equation*}
\forall z\in\C, \qquad	\vert h(z)\vert \leq C\, e^{a\Re(z)^2},
	\end{equation*} and
	\begin{equation*}
		\Vert h_{|\mathbb{R}}\Vert_{L_{\alpha}^p}=\left( \int_{\mathbb{R}}\vert h(x)\vert^p \vert x\vert^{2\alpha+1} dx \right)^{\frac{1}{p}}<+\infty,
	\end{equation*}
where $C$ and $a$ are positive constants.\\
Then $\quad h\equiv 0$ if $p\geq 1$, 
 and $ h$ is a constant on $\mathbb{C}$ if $p=+\infty$.
	\label{lem3.3}
\end{lem}  
\begin{proof}[Proof of Theorem \ref{Hardyth}]
Consider the functions $h_1$ and $h_2$, defined by:
\begin{equation}\label{h1h2}h_1(z)= e^{\frac{n}{4a}z^2}{\mathcal{T}}_1f(z),\quad and \qquad h_2(z)=e^{\frac{n}{4a}z^2}\mathcal{T}_2f(z).
\end{equation}
One can see, according to Lemma \ref{entire1} and Lemma \ref{entire2}, that $h_1$ and $h_2$ are entire functions on $\C$. Moreover from \eqref{Tineg} and \eqref{T2neg}, we have:
\begin{equation}
	\vert h_l(z)\vert \lesssim e^{\frac{n}{4a}\Re(z)^2}, \quad l=1,2. \label{h1_inq}
\end{equation} 
%Let . 
On the other hand, for $x\in\R$, we have $\mathcal{T}_1f(x)={\mathcal{F}}_{k,n}f_e(x^n)=\frac{1}{2}(\mathcal{F}_{k,n}f(x^n)+\mathcal{F}_{k,n}f(-x^n))$ and $\mathcal{T}_2f(x)={\mathcal{F}}_{k,n}f_o(x^n)=\frac{1}{2}(\mathcal{F}_{k,n}f(x^n)-\mathcal{F}_{k,n}f(-x^n))$. Condition \eqref{C2} implies that $$e^{nbx^2}\left| \mathcal{F}_{k,n}f(x^n)\right|\leq  C\quad and \quad e^{nbx^2}\left| \mathcal{F}_{k,n}f(-x^n)\right|\leq  C.$$ Hence $$\left|e^{nbx^2} \mathcal{T}_lf(x)\right|\leq C,\quad l=1,2.$$
Consequently, for $p\in[1,+\infty)$ and $ab>\frac{1}{4}$,
\begin{equation*}
\int_{\mathbb{R}}\vert h_l(x)\vert^p \vert x\vert^{2kn-n+1} dx\lesssim \int_{\R} e^{-np(b-\frac{1}{4a})x^2} \vert x\vert^{2kn-n+1}dx<+\infty.
\end{equation*}
Lemma \ref{lem3.3} implies that $h_1$ and $h_2$ are identically zero. 
That is ${\mathcal{F}}_{k,n}f_e(.^n)=0$ and ${\mathcal{F}}_{k,n}f_o(.^n)=0$.
Specifically, for all $x\in(0,+\infty)$,
$${\mathcal{F}}_{k,n}f_e(x^n)=0 \quad and \quad {\mathcal{F}}_{k,n}f_o(x^n)=0$$
As the mapping $x\mapsto x^n$ establishes a bijection from $\R_+$ into itself, it follows that for all $x\in(0,+\infty)$,
$${\mathcal{F}}_{k,n}f_e(x)=0\quad and \quad {\mathcal{F}}_{k,n}f_o(x)=0.$$ Since \begin{equation}\label{FeFo}{\mathcal{F}}_{k,n}f_e=\left( {\mathcal{F}}_{k,n}f\right) _e\quad and \quad {\mathcal{F}}_{k,n}f_o=\left( {\mathcal{F}}_{k,n}f\right)_o,                                            \end{equation} 
we deduce that for all $x\in\R$,
$$\left( {\mathcal{F}}_{k,n}f\right) _e(x)=0 \quad and \quad \left( {\mathcal{F}}_{k,n}f\right) _o(x)=0.$$
Therefore $$ {\mathcal{F}}_{k,n}f =0,$$
which leads to $f=0$ a.e.

{2. The case $a.b=\frac{1}{4}$.} The inequalities \eqref{h1_inq} remain valid.
In accordance with condition \eqref{C2}, we obtain
 $$\vert h_1(x)\vert= \vert{\mathcal{F}}_{k,n}f_e(x^n)e^{\frac{n}{4a}x^2}\vert\leq C \quad\text{ and }\quad\vert h_2(x)\vert= \vert{\mathcal{F}}_{k,n}f_o(x^n)e^{\frac{n}{4a}x^2}\vert\leq C$$
 Applying Lemma \ref{lem3.3}, we conclude that:
 $$ h_1(z)=\lambda_1 \quad\text{ and }\qquad h_2(z)=\lambda_2,$$
 where $\lambda_1,\,\lambda_2\in \C$.
 Thus, 
  $$ \mathcal{T}_1f(z)=\lambda_1 e^{-nbz^2}\quad\text{ and }\qquad \mathcal{T}_2f(z)=\lambda_2e^{-nbz^2}.$$
In particular, for all $x\in\mathbb{R}$,
 \begin{equation}\label{last}\left( \mathcal{F}_{k,n}f \right) _e(x^n)={\mathcal{F}}_{k,n}f_e(x^n)=\lambda_1 e^{-nbx^2}\quad 
  \text{and }\left( {\mathcal{F}}_{k,n}f \right)_o(x^n)={\mathcal{F}}_{k,n}f_o(x^n)=\lambda_2e^{-nbx^2}.
\end{equation}
Note that if $n$ is an odd integer then ${\mathcal{F}}_{k,n}f_o(x^n)$ becomes an odd function and the equality \eqref{last} holds only when $\lambda_2=0$. Consequently,
$$\forall\, x\in \R,\qquad{\mathcal{F}}_{k,n}f(x^n)=\lambda_1 e^{-nb|x|^2}.$$
Using that the mapping $x\mapsto x^n$  establishes a bijection from $\R$ into itself for $n$ odd integer,
\begin{equation}\label{oddab=1/4}\mathcal{F}_{k,n}f(x)=\lambda_1 e^{-nb\vert x\vert^\frac{2}{n}}.
\end{equation}
 By inverting \eqref{oddab=1/4} and from relation \eqref{FGaussian}, we obtain
 $$f(x)= C\, e^{-na\vert x\vert^{\frac{2}{n}}} \quad a.e.$$ 
 {If $n$ is even integer}, \eqref{last} still true for $x>0$, which yield  
$$\left( \mathcal{F}_{k,n}f \right) _e(x)=\lambda_1 e^{-nb\vert x\vert^\frac{2}{n}}\quad
\text{ and }\left( {\mathcal{F}}_{k,n}f \right)_o(x)=\lambda_2e^{-nb\vert x\vert^\frac{2}{n}}.$$
Thus, for all $x \in \mathbb{R}$, we derive that
$$\mathcal{F}_{k,n}f (x)={\mathcal{F}}_{k,n}f_e(x)+{\mathcal{F}}_{k,n}f_o(x)=(\lambda_1+sgn(x)\lambda_2)e^{-nb\vert x\vert^\frac{2}{n}}.$$
By applying the inverse formula, we obtain
$$f_e(x)= C\, e^{-na\vert x\vert^{\frac{2}{n}}} \quad a.e.$$ 
and $$
f_o(x) =\lambda_2\, \frac{(-i)^n n}{\Gamma(\alpha+n+1)}\left(\frac{n}{2} \right)^{\alpha+n}\int_{0}^{+\infty}e^{-nbu^2} xu^n j_{\alpha+n}(nu\vert x\vert^{\frac{1}{n}}u) u^{2\alpha+1}du.$$
Using formula \cite[p.394]{Wa}, we get
$$ f_o(x)=\lambda_2x\frac{(-i)^n \Gamma(kn+1) n^{kn+\frac{n}{2}}}{2(nb)^{\frac{1}{2}(kn-\frac{n}{2}+2)}\Gamma(kn+\frac{n}{2}+1)}\left(\sqrt{na}\right)^{kn+\frac{n}{2}} e^{-na\vert x\vert^{\frac{2}{n}}}\, _1F_1(\frac{n}{2};kn+\frac{n}{2}+1; na\vert x\vert^{\frac{2}{n}}),$$
where $_1F_1$ is the confluent hypergeometric function. Since $\frac{n}{2}$, $kn+\frac{n}{2}+1$ and $na\vert x\vert^{\frac{2}{n}}$ are positive, it follows that $$_1F_1(\frac{n}{2};kn+\frac{n}{2}+1; na\vert x\vert^{\frac{2}{n}})\geq \; _1F_1(\frac{n}{2};kn+\frac{n}{2}+1; 0)= 1.$$ This implies that $f_o$ satisfies condition \eqref{C1} of Hardy's theorem if and only if $\lambda_2=0$.
\\ Consequently, we conclude that when $ab=\frac{1}{4}$,  the only functions that can be controlled by a Gaussian-type function, along with their  $(k,\frac{2}{n})-$Fourier transform, are those of the form
$f(x)=C\, e^{-na|x|^\frac{2}{n}}$.\\
 3.  When $ab<\frac{1}{4}$, we take $a<\delta<\frac{1}{4b}$ and we consider the family of functions $f_\delta(x)=e^{-\delta n\vert x\vert^{\frac{2}{n}}} $. These functions satisfy the conditions \eqref{C1} and \eqref{C2}.
\end{proof}
 
\section{A dynamical version of Hardy's uncertainty principle}

As in the classical case, we introduce a Sobolev-type space to utilize the $(k,\frac{2}{n})$-Fourier transform in solving \eqref{heateq}. Let $W_{k,n}^2$ denote the Sobolev space constructed via the operator $|x|^{2-\frac{2}{n}}\Delta_k$, defined as the subspace of $L^2_{k,n}$ such that $|x|^{2-\frac{2}{n}}\Delta_{k}f\in L^2_{k,n}$. $$W_{k,n}^2 = \left\{ f \in L^2_{k,n} \;:\; |x|^{2-\frac{2}{n}}\, \Delta_k f \in L^2_{k,n} \right\}.$$ Notably, the operator $|x|^{2-\frac{2}{n}}\Delta_k$ can be introduced in $L^2_{k,n}$ under the condition $2k + \frac{2}{n} - 2 > 0$, by
\begin{equation}
	\mathcal{F}_{k,n} \big( \vert x\vert^{2-\frac{2}{n}} \Delta_k \big) = -\vert x\vert^\frac{2}{n} \circ \mathcal{F}_{k,n}.
	\label{DeltaF}
\end{equation}
This property, along with others arising from the representation-theoretic construction of $\mathcal{F}_{k,n}$, can be found in \cite[Theorem 5.6]{BKO}. Hence, applying $\mathcal{F}_{k,n}$ to the heat equation \eqref{heateq}, we obtain that $ \mathcal{F}_{k,n}(u_t)$ satisfies the one order differential equation:
\begin{equation}
 		\label{2}
 		\partial_t \mathcal{F}_{k,n}(u_t)(\xi)=-n\vert \xi\vert^\frac{2}{n}\mathcal{F}_{k,n}(u_t)(\xi).
 	\end{equation}
If we consider the initial condition $$u_0(x)=u(0,x)\in L^2_{k,n}(\mathbb{R}),$$ then we get \begin{equation}
\mathcal{F}_{k,n}(u_t)(\xi)=e^{-n\vert \xi\vert^\frac{2}{n}t}\mathcal{F}_{k,n}u_0(\xi).
\label{F_ut}
\end{equation}
Invoking relation \eqref{FGaussian}, we assert, for $t>0$, that 
 	$$\mathcal{F}_{k,n}\left(\left(\frac{1}{2t} \right)^{k n-\frac{n}{2}+1}e^{-\frac{n}{4t}\vert .\vert^\frac{2}{n}}  \right)(\xi) =e^{-n\vert \xi\vert^\frac{2}{n}t}.$$
Consequently using the convolution structure, studied in \cite{BeNe, BeNe2}, which holds for $f\in L^1_{k,n}(\mathbb{R})$ and $g\in L^2_{k,n}(\mathbb{R})$, as
\begin{equation*}
	\mathcal{F}_{k,n}\left( f\star_{k,n}g \right) =\mathcal{F}_{k,n}\left( f\right)\mathcal{F}_{k,n}\left( g \right),
\end{equation*}
we infer $$\mathcal{F}_{k,n}(u_t)(\xi)=\mathcal{F}_{k,n}\left(\left(\frac{1}{2t} \right)^{k n-\frac{n}{2}+1}e^{-\frac{n}{4t}\vert x\vert^\frac{2}{n}}  \star_{k,n}u_0\right)(\xi).$$
Finally, by applying the inversion formula \eqref{Finverse}, we deduce that a solution of the heat equation \eqref{heateq} takes the form
\begin{equation}
u(t,x)=\left(\frac{1}{2t} \right)^{k n-\frac{n}{2}+1}e^{-\frac{n}{4t}\vert x\vert^\frac{2}{n}}  \star_{k,n}u_0((-1)^n x),\quad t>0.
\end{equation}

The dynamical version of Hardy theorem invoking heat operator states as follows. 
\begin{thm}\label{DHardyth}
 	Let $u\in C^1([0,T],W^2_{k,n})$  be a solution of the heat equation 
 	 \begin{equation}\label{1}
 		\begin{cases}
 			H_{k,n}u(t,x)=0\\
 			u_0(x)=u(0,x)\in L^2_{k,n}(\mathbb{R}), \qquad (t,x)\in [0,T]\times\mathbb{R}.
 		\end{cases}
 	\end{equation}
Suppose that $u_0(x)\in L^1_{k,n}(\mathbb{R})$ and 
\begin{eqnarray}\label{u_T}
\vert u(T,x)\vert \leq C e^{-n\delta\vert x\vert^{\frac{2}{n}}}.
\end{eqnarray} Then : 
if $\delta\geq \frac{1}{4T}$ then $u=0$.
\end{thm}
\begin{proof} 
To enhance readability, we denote by $u_t$ the function $ x\longmapsto u(t,x)$. \\
Note that \eqref{u_T} ensures that $u_T$ satisfies the first condition \eqref{C1} of Hardy's theorem. Moreover, combining equation \eqref{F_ut} with inequality \eqref{Fourier borne}, we obtain
$$\mathcal{F}_{k,n}(u_T)(\xi)\leq C\,\Vert u_0\Vert_{L^1_{k,n}}e^{-n\vert \xi\vert^\frac{2}{n}T}.$$
Hence the function $u_T$ satisfies \eqref{C2}, which implies, due to Theorem \ref{Hardyth}, that
 { if $\delta >\frac{1}{4T}$ } then $u_T=0$. By examining equation \eqref{F_ut}, we deduce that the cancellation of $u_T$ at time $T$ means the initial function $u_0$ must also be zero. Thus, this condition propagates to all times, which implies that $u_t=0$  for all $t\in[0,T].$
 
On the other hand, when $\delta=\frac{1}{4T}$, we find that $$u_T(x)=\lambda e^{-n\delta\vert x\vert^\frac{2}{n}},$$  and 
$$\mathcal{F}_{k,n}u_T(x)=\lambda\,(2T)^{k n-\frac{n}{2}+1}e^{-nT|x|^\frac{2}{n}}.$$
\eqref{F_ut} allows us to determine that 
 	$$\mathcal{F}_{k,n}u_0(x)=\lambda\,(2T)^{k n-\frac{n}{2}+1}.$$
Since $u_0$ belongs to $L^2_{k,n}$, by Plancherel formula, it will be the same for $\mathcal{F}_{k,n}u_0$, which follows that $\lambda=0$. Thus $\mathcal{F}_{k,n}u_0=0$ and $u_0=0 \; a.e.$
 \end{proof}
% \begin{thm}Theorem \ref{Hardyth}  is equivalent to Theorem  \ref{DHardyth}. \end{thm}
\begin{rem}
In the proof of Theorem \ref{DHardyth}, it is clear that Hardy's Theorem \ref{Hardyth} implies Theorem \ref{DHardyth}. Reversely, 
let $f$ be a function that satisfies the conditions \eqref{C1} and \eqref{C2} of Hardy's theorem. Consider the function $u(t,x)$ defined for all $t\geq 0$ and $x\in \R$, by
\begin{equation}
	\mathcal{F}_{k,n}u_t(x)=f(x)\,e^{-nt\vert x\vert^{\frac{2}{n}}}.
	\label{Fu=f}
\end{equation} 
\eqref{Fu=f} and \eqref{C1} lead to $\mathcal{F}_{k,n}u_t$ and $|\,.\,|^\frac{2}{n}\mathcal{F}_{k,n}u_t$ belong to $ L^2_{k,n} $, which implies by the inversion formula that $u_t$ belongs to the Sobolev-type space $W_{k,n}^2$, and $u\in C^1([0,T],W_{k,n}^2)$, where $T>0$. Note also that by virtue of \eqref{C2}, $u_0=\mathcal{F}_{k,n}^{-1}f=\mathcal{F}_{k,n}f((-1)^n.)$ belongs to $ L^2_{k,n}\cap L^1_{k,n} $. 

Moreover, applying the derivative with respect to $t$ to \eqref{Fu=f}, we obtain:
$$ \partial_t\mathcal{F}_{k,n}u_t(x)=n\mathcal{F}_{k,n}(\vert x\vert^{2-\frac{2}{n}}\Delta_k u_t)(x).$$
Then, by the $(k,\frac{2}{n})-$Fourier inversion formula, we obtain that $u$ is a solution of \eqref{1}. Specifically, for $T>0$,  $$u(T,x)=\frac{1}{\left(2T\right)^{k n-\frac{n}{2}+1}}e^{-\frac{n}{4T}\vert x\vert^\frac{2}{n}}  \star_{k,n}u_0((-1)^n x).$$
Using Young inequality \cite{Boubatra_Negzaoui_Sifi} $$|u(T,x)|\lesssim \|e^{-\frac{n}{4T}\vert x\vert^\frac{2}{n}}\|_{L^1_{k,n}}  \| u_0\|_{L^\infty_{k,n}}. $$
Since $$\| u_0\|_{L^\infty_{k,n}}=\| \mathcal{F}_{k,n}^{-1}f\|_{L^\infty_{k,n}}=\| \mathcal{F}_{k,n}f\|_{L^\infty_{k,n}}\lesssim e^{-nb|x|^\frac{2}{n}}$$ then we infer according to Theorem \ref{DHardyth}, when $T=a>0$, that :
if $ab>\frac{1}{4}$ then $u_t=0$. So, by \eqref{Fu=f}, we derive that $f=0$. %\\If $ab=\frac{1}{4}$ then $u(t,x)=\lambda\left(\frac{a}{t} \right)^{kn-\frac{n}{2}+1}\,e^{-\frac{n}{4t}\vert x\vert^{\frac{2}{n}} }$.Therefore  $$\mathcal{F}_{k,n}u_t(x)=f(x)\,e^{-nt\vert x\vert^{\frac{2}{n}}}=\lambda\left(\frac{a}{t} \right)^{kn-\frac{n}{2}+1}\,\mathcal{F}_{k,n}(e^{-\frac{n}{4t}\vert x\vert^{\frac{2}{n}}})=\lambda\,(2a)^{kn-\frac{n}{2}+1}e^{-nt\vert x\vert^{\frac{2}{n}}}$$

This observation confirms the interaction between Hardy's theorem and dynamic systems.

\end{rem}
\section{$L^p-L^q$ versions of Hardy's Theorem}
\subsection{Miyachi's theorem}
In contrast, Miyachi's theorem provides a more flexible yet stricter interpretation of the uncertainty principle: instead of imposing exact Gaussian decay, it requires a logarithmic integrability condition on the Fourier side.

Let's first present a Phragm\'en-Lindel\"of type lemma.
\begin{lem}\cite{Mi}\label{Phragmen_Miyachi}
	Let  $k\geq \frac{n-1}{2n}$ and  $h$ be an entire function on $\mathbb{C}$ such that:
	\begin{equation}
		\vert h(z) \vert \leq C e^{a \Re(z)^2}
		\label{LC1}
	\end{equation}
	\begin{equation}
		\int_{\mathbb{R}}log^+(\vert h(x)\vert)\vert x\vert^{2kn-n+1}dx<+\infty
		\label{LC2}
	\end{equation}
	for some positive constants $C$ and $a$, and where $log^+ r= \left\{ \begin{array}{cc}
		log(r) &\text{ if }r>1\\  0&\text{else}
	\end{array}
	\right..$\\
	Then $h$ is a constant function.
\end{lem}
 \begin{lem}\label{entireT_Miyachi}
	Let $p,q\in[1,+\infty]$ and $a>0$. Suppose that $f$ is a measurable function on $\mathbb{R}$ satisfying
	\begin{equation}\label{C_lemma_Miyachi}
		e^{na|x|^{\frac{2}{n}}} f(x) \in L^p_{k,n} + L^q_{k,n}.
	\end{equation}
	Then ${\mathcal{T}}_{1}f$ and ${\mathcal{T}}_{2}f$ are well-defined and extend to an entire function on $\mathbb{C}$. 
	Furthermore, for every $z\in\mathbb{C}$, one has
	\begin{equation}\label{T_ing_miyachi}
		\vert \mathcal{T}_lf(z)\vert \lesssim 
		e^{\tfrac{n}{4a}\,(\Im (z))^2}, \quad l=1,2.
	\end{equation}
\end{lem}
\begin{proof}
	Assume that $f$ satisfies condition \eqref{C_lemma_Miyachi}. 
	Then there exist measurable functions $f_1$ and $f_2$ such that 
	$$   e^{na|\,\cdot\,|^{\tfrac{2}{n}}} f_1 \in L^p_{k,n}, 
	\quad 
	e^{na|\,\cdot\,|^{\tfrac{2}{n}}} f_2 \in L^q_{k,n}  \qquad \text{ and }\qquad f =f_1+f_2  $$
	Consequently, $f_1$ and $f_2$ satisfy the hypotheses of Lemma~\ref{entire1} and Lemma~\ref{entire2}. Moreover,
	\[
	\mathcal{T}_{1}f = \mathcal{T}_{1}(f_1) + \mathcal{T}_{1}(f_2)
	\quad \text{and} \quad
	\mathcal{T}_{2}f = \mathcal{T}_{2}(f_1) + \mathcal{T}_{2}(f_2).
	\]
	Therefore, $\mathcal{T}_{1}f$ and $\mathcal{T}_{2}f$ inherit the conclusion of the lemma.
	
\end{proof}

%This motivates the establishment of Miyachi's theorem for the $(k,\frac{2}{n})-$Fourier transform, presented in the following:
\begin{thm}\label{th_Miyachi}
	Let $f$ be a mesurable function on $\mathbb{R}$ such that 
	\begin{equation}
		e^{na\vert x\vert^{\frac{2}{n}}}f\in L^p_{k,n}+L^q_{k,n}
		\label{C1_Miyachi}
	\end{equation} and
	\begin{equation}
		\int_{\mathbb{R}}\log^+\left|  \frac{e^{nb\vert x\vert^\frac{2}{n}}\mathcal{F}_{k,n}f_e(x) }{C}\right|  d\mu_{k,n}(x)+\int_{\mathbb{R}}\log^+\left|  \frac{e^{nb\vert x\vert^\frac{2}{n}}\mathcal{F}_{k,n}f_o(x) }{C}\right|  d\mu_{k,n}(x)<+\infty
		\label{C2_Miaychi}
	\end{equation}
	for some constants $a, b, C > 0$ and $p,q\in [1,+\infty]$. Then
	\begin{enumerate}
		\item If $a.b>\frac{1}{4}$, then $f=0$.
		\item If $a.b=\frac{1}{4} $, then $f(x)=\lambda e^{-na\vert x\vert^\frac{2}{n}}$, $\quad \vert \lambda\vert \leq (2a)^{kn - \frac{n}{2} + 1} . C$.
		\item If $a.b<\frac{1}{4}$, then there exist multiple functions satisfying the given constraints.
	\end{enumerate}
\end{thm}
\begin{proof}
	Let $a.b\geq\frac{1}{4}$. 
We make use of the functions $h_1$ and $h_2$
previously introduced in equation \eqref{h1h2}.
	Then, using inequality \eqref{T_ing_miyachi}, we get:
	\begin{align*}
		\vert h_l(z)\vert \lesssim   e^{\frac{n}{4a}(Re(z))^2},\qquad l=1,\,2.
	\end{align*}
On the other hand, since $\frac{1}{4a}\leq b$, we have 
$$
		\int_{\mathbb{R}}\log^+\left|  \frac{h_2(x)}{C}\right| \vert x\vert^{2kn-n+1}dx \leq \int_{\mathbb{R}}\log^+\left|  \frac{e^{nb x^2}\mathcal{F}_{k,n}f_o(x^n)}{C} \right|  \vert x\vert^{2kn-n+1}dx.
$$
Since the change of variable $u=x^n$ is valid only on $(0,+\infty)$ for $n\in\mathbb{N}$, we may split the integral into $\int_0^{+\infty}$ and $\int_{-\infty}^0$ and consider $t=-x$, we find, using \eqref{FeFo},
$$\int_{\mathbb{R}}\log^+\left|  \frac{e^{nb x^2}\mathcal{F}_{k,n}f_o(x^n)}{C} \right|  \vert x\vert^{2kn-n+1}dx\lesssim \int_0^{+\infty}\log^+\left|  \frac{e^{nb\vert x\vert^\frac{2}{n}}\mathcal{F}_{k,n}f_o(x) }{C}\right| d\mu_{k,n}(x)$$
	By virtue of \eqref{C2_Miaychi}, we have
$$\int_0^{+\infty}\log^+\left|  \frac{e^{nb\vert x\vert^\frac{2}{n}}\mathcal{F}_{k,n}f_o(x) }{C}\right| d\mu_{k,n}(x)<+\infty.
$$
Consequently
\begin{equation}\label{h1}
	\int_{\mathbb{R}}\log^+\left|  \frac{h_2(x)}{C}\right| \vert x\vert^{2kn-n+1}dx \lesssim\int_{\mathbb{R}}\log^+\left|  \frac{e^{nb x^2}\mathcal{F}_{k,n}f_o(x^n)}{C} \right|  \vert x\vert^{2kn-n+1}dx <+\infty
\end{equation}
Then, by similar arguments, we obtain
\begin{equation}\label{h2}
 \int_{\mathbb{R}}\log^+\left|  \frac{h_1(x)}{C}\right| \vert x\vert^{2kn-n+1}dx \lesssim\int_{\mathbb{R}}\log^+\left|  \frac{e^{nb x^2}\mathcal{F}_{k,n}f_e(x^n)}{C} \right|  \vert x\vert^{2kn-n+1}dx <+\infty
\end{equation}
Noting that  $\frac{h_1}{C}$ and $\frac{h_2}{C}$  satisfy the assumptions of Lemma \ref{Phragmen_Miyachi}, we deduce that $h_1=C\,\lambda_1$ and $h_2=C\,\lambda_2$, where $\lambda_1,\lambda_2\in \mathbb{C}$. Thus %for $x\in \mathbb{R}$,
	\begin{equation}
		{\mathcal{F}}_{k,n}f_e(x^n)=C \lambda_1e^{-\frac{n}{4a}x^2},\text{ and }\quad{\mathcal{F}}_{k,n}f_o(x^n)=C \lambda_2e^{-\frac{n}{4a}x^2},
		\label{Forme_h_1=of_Ff}
	\end{equation} 
$\bullet$ If $ab>\frac{1}{4}$, then combining  \eqref{Forme_h_1=of_Ff} and \eqref{h1}, we get $\lambda_2=0$. Same argument using \eqref{Forme_h_1=of_Ff} and \eqref{h2} allows to deduce $\lambda_1=0$.
By proceeding similarly to the proof of Theorem \ref{Hardyth}, we conclude that $f = 0$ almost everywhere.
\\ $\bullet$ 
If $ab = \tfrac{1}{4}$, then relations \eqref{Forme_h_1=of_Ff}, \eqref{h1},  and  \eqref{h2} lead to $\vert \lambda_1\vert \leq C$ and $\vert \lambda_2\vert \leq C$, where $C$ is the constant provided by the condition \eqref{C2_Miaychi}. 
Following the procedure in 2. of the proof of Theorem \ref{Hardyth}, step by step, we find  that $\lambda_2=0$ and
%$$ \lambda_2\,\,_1F_1\left(\frac{n}{2}; kn + \frac{n}{2} + 1; na\vert x\vert^{\frac{2}{n}}\right)\, x\; \in L^p_{k,n}+L^q_{k,n},$$	which holds if and only if $\lambda_2=0 $. This constraint is necessary for 	$f$ to still satisfy condition \eqref{C1_Miyachi}. Therefore,
$$ f(x) = (2a)^{kn - \frac{n}{2} + 1}\lambda_1\,e^{-na\vert x\vert^{\frac{2}{n}}}, \qquad\text{ with }\vert \lambda_1\vert <C.$$
$\bullet$ If $ab < \tfrac{1}{4}$, then    
we choose $a<\delta<\frac{1}{4b}$ and we prove that the family of functions $$f_\delta(x)=\lambda e^{-\delta n\vert x\vert^{\frac{2}{n}}} ,\; \vert\lambda\vert < C,$$ satisfies conditions \eqref{C1_Miyachi} and \eqref{C2_Miaychi}.

\end{proof}
\subsection{Cowling Price theorem}
As an application of Miyachi's theorem, we recover an $L^p-L^q$ version of Hardy's theorem, commonly referred to as the Cowling--Price theorem.

\begin{thm} \label{thCowlingPrice}Consider  $a,b>0$, $1\leq p,q\leq +\infty$ such that $\min(p,q)<+\infty$. Lef $f$ be a measurable function on $\mathbb{R}$ satisfying 
	\begin{equation}
		\Vert e^{na\vert x\vert^{\frac{2}{n}}}f\Vert_{L^p_{k,n}}<+\infty,
		\label{C1_cowling}
	\end{equation} and 
	\begin{equation}
		\Vert e^{nb\vert x\vert^{\frac{2}{n}}}\mathcal{F}_{k,n} f\Vert_{L^q_{k,n}}<+\infty. \label{C2_cowling} 
	\end{equation}
	Then we have 
	\begin{enumerate}
		\item[1.] If $a.b \geq \frac{1}{4},$ then $f=0$ almost everywhere.
		\item[2.]  If $a.b<\frac{1}{4},$ then there exist infintely many linearly independ functions satisfying the conditions \eqref{C1_cowling} and \eqref{C2_cowling}.
	\end{enumerate}
\end{thm}

\begin{proof}
	Assume that $f$ satisfies the hypotheses of Theorem \ref{thCowlingPrice}. 
	Since
	$$
	L^p(d\mu_{k,n}) \subset L^1(d\mu_{k,n}) + L^\infty(d\mu_{k,n}),
	$$
	it follows that $f$ fulfills condition~\eqref{C1_Miyachi}. 
	Moreover, using the elementary bound
	$$
	\log^+|x| \leq |x|^q, \qquad x\in\R,
	$$ and according to \eqref{FeFo}, 	we obtain $|\mathcal{F}_{k,n}f_e(x)|\leq |\mathcal{F}_{k,n}f(x)|$ and $|\mathcal{F}_{k,n}f_o(x)|\leq |\mathcal{F}_{k,n}f(x)|$. Hence
	\begin{equation}
		\int_{\R} \log^+\!\left( \frac{e^{nb|x|^{2/n}}\,|\mathcal{F}_{k,n}f_e(x)|}{C} \right) 
		d\mu_{k,n}(x) 
		\leq \frac{1}{C^q} \int_{\R} e^{qnb|x|^{2/n}} |\mathcal{F}_{k,n}f(x)|^q \, d\mu_{k,n}(x) 
		< \infty
	\end{equation}
	and \begin{equation}
		\int_{\R} \log^+\!\left( \frac{e^{nb|x|^{2/n}}\,|\mathcal{F}_{k,n}f_o(x)|}{C} \right) 
		d\mu_{k,n}(x) 
		\leq \frac{1}{C^q} \int_{\R} e^{qnb|x|^{2/n}} |\mathcal{F}_{k,n}f(x)|^q \, d\mu_{k,n}(x) 
		< \infty.
	\end{equation}
Thus, the assumptions of Theorem \ref{th_Miyachi} are satisfied. In particular, if $ab > \tfrac{1}{4}$, we necessarily obtain $f = 0$ almost everywhere. When $ab < \tfrac{1}{4}$, there exist infinitely many nontrivial functions satisfying simultaneously \eqref{C1_cowling} and \eqref{C2_cowling}. In the critical case $ab = \tfrac{1}{4}$, the function $f$ must be of the form
$
f(x) = \lambda e^{-na\vert x\vert^{2/n}}
$
for some constant $\lambda$. However, the condition \eqref{C1_cowling} forces $\lambda = 0$, and therefore $f = 0$.

\end{proof}
%\subsection*{Disclosure statement}No potential conflict of interest was reported by the authors.
%\bibliographystyle{abbrv}  % Choose a bibliography style (e.g., plain, abbrv, alpha)
%\bibliography{Jilani_Negzaoui}

\end{document}